\documentclass[a4paper,11pt]{article}
\usepackage[english]{babel}
\usepackage[utf8]{inputenc}
\usepackage{paralist,url,verbatim,anysize}
\usepackage{amscd}
\usepackage{mathrsfs}
\usepackage{microtype}
\usepackage{manfnt}
\usepackage{nicefrac}

\usepackage[nodisplayskipstretch]{setspace}
\usepackage{fancyhdr}
\fancypagestyle{plain}{%
\fancyhf{} 
\fancyfoot[C]{\fontsize{12pt}{12pt}\selectfont\thepage} 

}
\usepackage[e]{esvect}
\usepackage{amsmath,amsthm,amssymb,amsfonts}
\usepackage{graphicx, graphics}
\usepackage[bookmarksopen=true]{hyperref}
\hypersetup{colorlinks=true, citecolor=blue}
\usepackage{epstopdf}
\usepackage[mathlines, pagewise]{lineno}
\usepackage{bbm}
\usepackage[font=small,labelfont=bf]{caption}
\usepackage{subcaption}
\usepackage[symbol*]{footmisc}
\theoremstyle{plain}
\newtheorem{theorem}{\bf Theorem}[subsection]

\newtheorem{cor}[theorem]{Corollary}
\newtheorem{lemma}[theorem]{Lemma}
\newtheorem{proposition}[theorem]{Proposition}

\newtheorem{thmnonumber}{\bf Theorem}
\newtheorem{cornonumber}[thmnonumber]{\bf Corollary}

\theoremstyle{definition}
\newtheorem{rem}[theorem]{Remark}
\newtheorem{definition}[theorem]{Definition}

\DeclareMathAlphabet{\mathpzc}{OT1}{pzc}{m}{it}
\newenvironment{frcseries}{\fontfamily{frc}\selectfont}{}

\newcommand{\NE}{\operatorname{NE} }
\newcommand{\TT}{\mathrm{T}}
\newcommand{\F}{\mathrm{F}}
\newcommand{\rint}{\mathrm{relint}}
\newcommand{\RN}{\mathrm{N}}

\newcommand{\R}{\mathbb{R}}
\renewcommand{\SS}{\mathbb{S}}
\newcommand{\RS}{\mathrm{R}\,}

\newcommand{\Lk}{\mathrm{Lk}\, }

\newcommand{\am}{\mathrm{arg} \mathrm{min}}
\newcommand{\St}{\mathrm{St}\, }

\newcommand{\sd}{\mathrm{sd}\, }

\newcommand{\CAT}{\mathrm{CAT}}
\newcommand{\cm}[1]{}

\DefineFNsymbolsTM{myfnsymbols}{
  \textasteriskcentered *
   \textparagraph \mathparagraph
  \textdaggerdbl \ddagger
  \textsection   \mathsection
  \textbardbl    \|%
}%
\setfnsymbol{myfnsymbols}
\renewcommand{\textparagraph}{** }
\begin{document}

\author{
Karim Adiprasito
\\
\small Einstein Institute for Mathematics\\ \small Hebrew University of Jerusalem\\
\small 91904 Jerusalem, Israel\\
\small \url{adiprasito@math.huji.ac.il}
\and
Bruno Benedetti \\
\small Department of Mathematics\\ \small University of Miami\\
\small 33146 Coral Gables, Florida\\
\small \url{bruno@math.miami.edu}}
\date{\small September 9, 2019}
\title{Collapsibility of CAT(0) spaces}
\maketitle
\bfseries

\mdseries


\begin{abstract}
Collapsibility is a combinatorial strengthening of contractibility. We relate this property to metric geometry by proving the collapsibility of any complex that is $\CAT(0)$ with a metric for which all vertex stars are convex. This strengthens and generalizes a result by Crowley. Further consequences of our work are: 
\begin{compactenum}[\rm (1)]
\item All $\CAT(0)$ cube complexes are collapsible.
\item Any triangulated manifold admits a $\CAT(0)$ metric if and only if it admits collapsible triangulations.
\item All contractible $d$-manifolds ($d \ne 4$) admit collapsible $\CAT(0)$ triangulations. This discretizes a classical result by Ancel--Guilbault. 
\end{compactenum}

\end{abstract}


\section{Introduction}


Whitehead's ``simple homotopy theory", a combinatorial approach to homotopy, was partly motivated by Poincar\'{e}'s conjecture that homotopy recognizes spheres among all closed manifolds. In contrast, Whitehead discovered that homotopy does not recognize $\R^n$ among all (open) manifolds. In fact, for each $d \ge 4$, there are compact contractible smooth $d$-manifolds whose boundary is not simply-connected~\cite{Mazur, Newman2, Whitehead1935}.

In 1939, Whitehead introduced \emph{collapsibility}, a combinatorial version of the notion of contractibility. All collapsible complexes are contractible, but the converse is false~\cite{BING, Zeeman}. The collapsibility property lies at the very core of PL topology and is of particular interest when applied to triangulations of manifolds. Using the idea of regular neighborhoods, Whitehead proved that the only manifold admitting collapsible PL triangulations is the ball \cite{Whitehead}.

Recently, collapsibility has regained interest in areas ranging from extremal combinatorics \cite{ALLM} to the study of subspace arrangements \cite{A, SalSet}; similar notions such as \emph{nonevasiveness} and \emph{dismantlability} have been explored in connection with theoretical computer science, cf.~e.g.~\cite{KahnSaksSturtevant, CivanYalcin, ChepoiBridge}.   Unfortunately, there are only a few available criteria to predict the collapsibility of a given complex:
\begin{compactenum}[(a)]
\item All cones are collapsible.
\item \emph{Chillingworth's criterion}~\cite{CHIL}: All subdivisions of convex $3$-polytopes are collapsible.
\item \emph{Crowley's criterion}~\cite{Crowley}: All $3$-dimensional pseudomanifolds that are $\CAT(0)$ with the equilateral flat metric (see below for the meaning), are collapsible.
\end{compactenum}
The main goal of this paper is to show that criterion (c) extends to all dimensions. In fact, in a separate companion paper, we will show that criterion (b) extends too~\cite{ABpart2}. Breaking the barrier of dimension three will allow us to list a few consequences at the interplay of combinatorics and metric geometry.

Crowley's criterion was advertised as a first ``\emph{combinatorial analog of Hadamard's theorem}''~\cite[p.~36]{Crowley}. Being $\CAT(0)$ is a property of metric spaces, popularized in the Eighties by Gromov~\cite{GromovHG}. Roughly speaking, $\CAT(0)$ spaces are metric spaces where any triangle formed by three geodesic segments looks thinner than any triangle in the euclidean plane formed by three straight lines of the same lengths. The classical Hadamard--Cartan theorem guarantees that simply connected, complete, locally $\CAT(0)$ spaces are (globally) $\CAT(0)$, and hence contractible.

Being ``$\CAT(0)$ with the equilateral flat metric'' is instead a property of simplicial complexes. It means that if we give the complex a piecewise-euclidean metric by assigning unit length to all its edges, then the complex becomes a $\CAT(0)$ metric space. Our first result is to prove the collapsibility of such complexes in all dimensions.  In fact, we show something stronger:

\begin{thmnonumber}[Theorem~\ref{thm:dischadamard}]
\label{mainthm:CATcollapsible}
Let $C$ be a simplicial complex that is $\CAT(0)$ with a metric for which all vertex stars are convex. Then $C$ is collapsible.
\end{thmnonumber}

{\em Convexity of vertex stars} means that for each vertex~$v$, with respect to the metric introduced on~$C$, the segment between any two points of~$\St (v,C)$ lies in~$\St (v,C)$. This condition cannot be removed, because not all triangulated $3$-balls are collapsible (although all of them are $\CAT(0)$ with a suitable metric). It is automatically satisfied by any complex in which all simplices are acute or right-angled. The same condition plays also a central role in the authors' proof of the Hirsch conjecture for flag polytopes \cite{ABHirsch}.

\begin{cornonumber}\label{mainthm:CATcube}
Every complex that is $\CAT(0)$ with the equilateral flat metric is collapsible. Similarly, every $\CAT(0)$ cube complex is collapsible.
\end{cornonumber}

For example, the ``space of phylogenetic trees'' introduced by Billera, Holmes and Vogtmann \cite{BilleraHolmesVogtmann} turns out to be collapsible (Corollary~\ref{cor:phylo}). A further application to enumerative combinatorics is the discrete Cheeger theorem in \cite{ABpart3}.

The converses of Theorem~\ref{mainthm:CATcollapsible} and Corollary~\ref{mainthm:CATcube} are false. Being $\CAT(0)$ with the equilateral flat metric is a more restrictive property than being collapsible; for example, any vertex-transitive complex with this property is a simplex (Proposition~\ref{prop:evasiveness}), while Oliver provided collapsible simplicial complexes with vertex-transitive symmetry group that are not the simplex, cf.\ \cite{KahnSaksSturtevant}. Nonetheless, a partial converse is possible if we are allowed to change the triangulation, as explained below.

Theorem~\ref{mainthm:CATcollapsible} allows us to understand the topology of manifolds that admit collapsible triangulations. This was investigated by Whitehead only in the PL case, where he proved that collapsible PL triangulations  are balls \cite{Whitehead}. Revisiting previous work of Ancel and Guilbault \cite{AS, AG2}, we are able  to prove the following:

\begin{thmnonumber}[Theorem~\ref{thm:ccc}, cf.\ \cite{AS, AG2}]\label{mainthm:CCC}
For any integer $d$, the $d$-manifolds admitting a collapsible triangulation are precisely those that admit a $\CAT(0)$ (polyhedral) metric. In particular, all contractible $d$-manifolds admit a collapsible triangulation, except for $d=4$. 
\end{thmnonumber}

For simplicial complexes, the situation changes slightly: paired with a result by the first author and Funar \cite[Proposition 15]{AFunar}, Theorem~\ref{mainthm:CATcollapsible} still implies that the complexes admitting a collapsible subdivision are precisely those that admit a polyhedral $\CAT(0)$ metric (Remark \ref{rem:Funar}). However, these are some, but not all, of the contractible complexes: For example, the Dunce Hat has no collapsible subdivision.



\section{Basic notions} \label{sec:Preliminaries}

\subsection{Geometric and intrinsic polytopal complexes}
By $\R^d$, $\mathbb{H}^d$ and $\SS^d$ we denote the euclidean $d$-space, the hyperbolic $d$-space, and the unit sphere in $\R^{d+1}$, respectively. A \emph{(euclidean) polytope} in $\R^d$ is the convex hull of finitely many points in $\R^d$. Similarly, a \emph{hyperbolic polytope} in $\mathbb{H}^d$ is the convex hull of finitely many points of $\mathbb{H}^d$. A \emph{spherical polytope} in $\SS^d$ is the convex hull of a finite number of points that all belong to some open hemisphere of $\SS^d$. Spherical polytopes are in natural one-to-one correspondence with euclidean polytopes, just by taking radial projections; the same is true for hyperbolic polytopes. A \emph{geometric polytopal complex} in $\R^d$ (resp.\ in $\SS^d$ or $\mathbb{H}^d$) is a finite collection of polytopes in $\R^d$ (resp.~$\SS^d$, $\mathbb{H}^d$) such that the intersection of any two polytopes is a face of both. An \emph{intrinsic polytopal complex} is a collection of polytopes that are attached along isometries of their faces (cf.\ Davis--Moussong~\cite[Sec.\ 2]{DavisMoussong}), so that the intersection of any two polytopes is a face of both. 

Two polytopal complexes $C,\, D$ are \emph{combinatorially equivalent}, denoted by $C\cong D$, if their face posets are isomorphic. Any polytope combinatorially equivalent to the $d$-simplex, or to the regular unit cube $[0,1]^d$, shall simply be called a \emph{$d$-simplex} or a \emph{$d$-cube}, respectively. A polytopal complex is \emph{simplicial}  (resp.~\emph{cubical}) if all its faces are simplices (resp.~cubes). The set of $k$-dimensional faces of a polytopal complex $C$ is denoted by $\F_k(C)$, and the cardinality of this set is denoted by $f_k(C).$ 

The \emph{underlying space} $|C|$ of a polytopal complex $C$ is the topological space obtained by taking the union of its faces. If two complexes are combinatorially equivalent, their underlying spaces are homeomorphic. We will frequently abuse notation and identify a polytopal complex with its underlying space, as is common in the literature. For instance, we do not distinguish between a polytope and the complex formed by its faces. If $C$ is simplicial, $C$ is sometimes called a \emph{triangulation} of $|C|$ (and of any topological space homeomorphic to $|C|$). If $|C|$ is isometric
to some metric space $M$, then $C$ is called a \emph{geometric triangulation} of~$M$.

A \emph{subdivision} of a polytopal complex $C$ is a polytopal complex $C'$ with the same underlying space of $C$, such that for every face $F'$ of $C'$ there is some face $F$ of $C$ for which $F' \subset F$. Two polytopal complexes $C$ and $D$ are called \emph{PL equivalent} if some subdivision $C'$ of $C$ is combinatorially equivalent to some subdivision $D'$ of $D$. In case $|C|$ is a topological manifold (with or without boundary), we say that $C$ is \emph{PL} (short for Piecewise-Linear) if the star of every face of $C$ is PL equivalent to the simplex of the same dimension.

A \emph{derived subdivision} $\sd  C$ of a polytopal complex $C$ is any subdivision of $C$ obtained by stellarly subdividing at all faces in order of decreasing dimension of the faces of $C$, cf.\ \cite{Hudson}.  An example of a derived subdivision is the \emph{barycentric subdivision}, which uses as vertices the barycenters of all faces of $C$.
 
If $C$ is a polytopal complex in $\R^d$ (or $\SS^d$, or $\mathbb{H}^d$) and $A$ is a subset of $\R^d$ (resp.~$\SS^d$, resp.~$\mathbb{H}^d$), we define the \emph{restriction $\RS(C,A)$ of $C$ to $A$} as the inclusion-maximal subcomplex $D$ of $C$ such that $D$ lies in $A$. The \emph{star} of $\sigma$ in $C$, denoted by $\St(\sigma, C)$, is the minimal subcomplex of $C$ that contains all faces of $C$ containing $\sigma$. The \emph{deletion} $C-D$ of a subcomplex $D$ from $C$ is the subcomplex of $C$ given by $\RS(C,  C{\setminus} \rint{D})$.

Next, we define the notion of \emph{link} with a metric approach. (Compare also Charney \cite{Charney} and Davis--Moussong~\cite[Sec.\ 2.2]{DavisMoussong}.) Let $p$ be any point of a metric space $X$. By $\TT_p X$ we denote the tangent space of $X$ at $p$. Let $\TT^1_p X$ be the restriction of $\TT_p X$ to unit vectors.  If $Y$ is any subspace of $X$, then $\RN_{(p,Y)} X$ denotes the subspace of the tangent space $\TT_p X$ spanned by the vectors orthogonal to $\TT_p Y$. If $p$ is in the interior of $Y$, we define $\RN^1_{(p,Y)} X:= \RN_{(p,Y)} X \cap \TT^1_p Y$.

If $\tau$ is any face of a polytopal complex $C$ containing a nonempty face $\sigma$ of $C$, then the set $\RN^1_{(p,\sigma)} \tau$ of unit tangent vectors in $\RN^1_{(p,\sigma)} |C|$ pointing towards $\tau$ forms a spherical polytope $P_p(\tau)$, isometrically embedded in $\RN^1_{(p,\sigma)} |C|$. The family of all polytopes $P_p(\tau)$ in $\RN^1_{(p,\sigma)} |C|$ obtained for all $\tau \supset \sigma$ forms a polytopal complex, called the \emph{link} of $C$ at $\sigma$; we will denote it by $\Lk_p(\sigma, C)$. If $C$ is a geometric polytopal complex in $X^d=\R^d$ (or $X^d=\SS^d$), then $\Lk_p(\sigma, C)$ is naturally realized in $\RN^1_{(p,\sigma)} X^d$. Obviously, $\RN^1_{(p,\sigma)} X^d$ is isometric to a sphere of dimension $d-\dim \sigma -1$, and will be considered as such. Up to ambient isometry $\Lk_p(\sigma, C)$ and  $\RN^1_{(p,\sigma)} \tau$ in $ \RN^1_{(p,\sigma)} |C|$ or $\RN^1_{(p,\sigma)} X^d$ do not depend on $p$; for this reason, $p$ will be omitted in notation whenever possible.

If $C$ is simplicial, and $v$ is a vertex of $C$, then $\Lk(v,C)$ is combinatorially equivalent to \[(C-v)\cap \St(v,C)=\St(v,C)-v.\]By convention, $\Lk(\emptyset, C):=C$. If $C$ is a simplicial complex, and $\sigma$, $\tau$ are faces of $C$, then $\sigma\ast \tau$ is the minimal face of $C$ containing both $\sigma$ and $\tau$ (assuming it exists). If $\sigma$ is a face of $C$, and $\tau$ is a face of $\Lk(\sigma,C)$, then $\sigma \ast \tau$ is the face of $C$ with $\Lk(\sigma,\sigma \ast \tau)=\tau$. In both cases, the operation~$\ast$ is called the \emph{join}.

\subsection{CAT(k) spaces and convex subsets} 
All of the metric spaces we consider here are \emph{compact, connected length spaces}. For a more detailed introduction, we refer the reader to the textbook by Burago--Burago--Ivanov~\cite{BuragoBuragoIvanov}. 
Given two points $a,b$ in a length space $X$, we denote by $|ab|$ the \emph{distance} between $a$ and $b$, which is also the infimum (and by compactness, the minimum) of the lengths of all curves from $a$ to $b$. Rectifiable curves that connect $a$ with $b$ and realize the distance $|ab|$ are \emph{(geodesic) segments} of $X$. A \emph{geodesic} from $a$ to $b$ is a curve $\gamma$ from $a$ to $b$ which is locally a segment. A \emph{geodesic triangle} in $X$ is given by three vertices $a,b,c$ connected by segments  $[a,b],[b,c]$ and $[a,c]$. 

Let $k$ be a real number. Depending on the sign of $k$, by \emph{the $k$-plane} we mean either the euclidean plane (if $k = 0$), or the sphere of radius $k^{-\frac{1}{2}}$ with its length metric (if $k > 0$), or the hyperbolic plane of curvature $k$ (if $k < 0$). A \emph{$k$-comparison triangle} for $[a,b,c]$ is a triangle $[\bar{a},\bar{b},\bar{c}]$ in the $k$-plane such that $|\bar{a}\bar{b}|=|ab|$, $|\bar{a}\bar{c}|=|ac|$ and $|\bar{b}\bar{c}|=|bc|$.  A length space $X$ ``\emph{has curvature $\le k$}'' if locally (i.e. in some neighborhood of each point of $X$) the following holds:
\smallskip
\begin{compactitem}[$\bullet$]
\item[] {\textsc{Triangle condition}:} For each geodesic triangle $[a,b,c]$ inside $X$ and for any point $d$ in the relative interior of $[a, b]$, one has  $|cd|\leq |\bar{c}\bar{d}|$, where $[\bar{a},\bar{b},\bar{c}]$ is any $k$-comparison triangle for $[a,b,c]$ and $\bar{d}$ is the unique point on $[\bar{a},\bar{b}]$ with $|ad| = |\bar{a}\bar{d}|$. 
\end{compactitem}
\smallskip
A \emph{$\CAT(k)$ space} is a length space of curvature $\le k$ in which the triangle condition holds globally, i.e.\ for any geodesic triangle (with edgelengths $<k^{\nicefrac{-1}{2}}\pi$ if $k$ is positive).
Obviously, any $\CAT(0)$ space has curvature $\le 0$. The converse is false: The  circle $S^1$ is non-positively curved (because it is locally isometric to $\mathbb{R}$) but not $\CAT(0)$, as shown by any geodesic triangle that is not homotopy equivalent to a point. Also, all $\CAT(0)$ spaces are contractible, while $S^1$ is not even simply connected. This is not a coincidence, as explained by the Hadamard--Cartan theorem, which provides a crucial local-to-global correspondence:

\begin{theorem}[Hadamard--Cartan theorem, cf.\ Alexander--Bishop \cite{AlexanderBishop}]
Let $X$ be any complete space with curvature $\le 0$. The following are equivalent:
\begin{compactenum}[\rm (1)] 
\item $X$ is simply connected;
\item $X$ is contractible;
\item $X$ is $\CAT(0)$.
\end{compactenum}
\end{theorem}

\noindent In a $\CAT(0)$ space, any two points are connected by a unique geodesic. The same holds for $\CAT(k)$ spaces ($k >0$), as long as the two points are at distance $ < \pi k^{-\frac{1}{2}}$.
 
Let $K$ be a subset of a metric space $X$. A set $K$ is called \emph{convex} if any two points of $K$ are connected by some segment in $X$ that lies entirely in $K$.

Let $c$ be a point of $X$ and let $K$ be a closed subset of $X$, not necessarily convex. We denote by $\pi_c(K)$ the subset of the points of $K$ at minimum distance from $c$, the {\em closest-point projection} of $c$ to $K$. In case $\pi_c(K)$ contains a single point, with abuse of notation we write $\pi_c(K) = x$ instead of $\pi_c(K) = \{x\}$. This is always the case when $K$ is convex, as the following well-known lemma shows. 

\begin{lemma}[{\cite[Prp.\ 2.4]{BH}}]\label{lem:unique}
Let $X$ be a connected $\CAT(k)$-space, $k\le 0$.  Let $c$ be a point of $X$. Then the function ``distance from $c$'' has a unique local minimum on each closed convex subset $K$ of $X$. Similarly, if $k>0$ and $K$ is at distance less than $\tfrac{1}{2}\pi k^{\nicefrac{-1}{2}}$ from $c$, then there exists a unique local minimum of distance at most $\tfrac{1}{2}\pi k^{\nicefrac{-1}{2}}$ from $c$.
\end{lemma}  

\subsection{Discrete Morse theory, collapsibility and non-evasiveness}

The \emph{face poset} $(C, \subseteq)$ of a polytopal complex $C$ is the set of nonempty faces of $C$, ordered with respect to inclusion. By $(\mathbb{R}, \le)$ we denote the poset of real numbers with the usual ordering. A \emph{discrete Morse function} is an order-preserving map $f$ from $(C, \subseteq)$ to $(\mathbb{R}, \le)$, such that the preimage $f^{-1}(r)$ of any number $r$ consists of either one element, or of two elements one contained in the other. A \emph{critical cell} of $C$ is a face at which $f$ is strictly increasing.  

The function $f$ induces a perfect matching on the non-critical cells: two cells are matched whenever they have identical image under $f$. This is called \emph{Morse matching}, and it is usually represented by a system of arrows: Whenever $\sigma \subsetneq \tau$ and $f(\sigma) = f(\tau)$, one draws an arrow from the barycenter of $\sigma$ to the barycenter of $\tau$. We consider two discrete Morse functions  \emph{equivalent} if they induce the same Morse matching. Since any Morse matching pairs together faces of different dimensions, we can always represent a Morse matching by its  associated partial function  
$\Theta$ from $ C$ to itself, defined on a subset 
of $C$ as follows:

\[
\Theta (\sigma) \ := \ 
\left\{
\begin{array}{cl}
\sigma & \textrm{if $\sigma$ is unmatched}, \\
\tau & \textrm{if $\sigma$ is matched with $\tau$ and $\dim \sigma < \dim \tau$}.
\end{array}
\right.
\]

A {\it discrete vector field} $V$ on a polytopal complex $C$ is a collection of pairs $(\sigma,\Sigma)$ of faces such that $\sigma$ is a codimension-one face of $\Sigma$, and no face of $C$ belongs to two different pairs of $V$. 
A \emph{gradient path} in $V$ is a concatenation of pairs of $V$
\[ (\sigma_0, \Sigma_0), (\sigma_1, \Sigma_1),  \ldots, (\sigma_k, \Sigma_k),\, k\ge 1,\]
so that for each $i$ the face $\sigma_{i+1}$ is a codimension-one face of $\Sigma_i$ different from $\sigma_i$. A gradient path is \emph{closed} if $\sigma_0 = \sigma_k$ for some $k$ (that is, if the gradient path forms a closed loop). A discrete vector field $V$ is a Morse matching if and only if $V$ contains no closed gradient paths~\cite{FormanADV}. 
The main result of discrete Morse Theory is the following:

\begin{theorem}[Forman {\cite[Theorem 2.2]{FormanUSER}}] \label{MorseThm}
Let $C$ be a polytopal complex. For any Morse matching on $C$, the complex $C$ is homotopy equivalent to a CW complex with one $i$-cell for each critical $i$-simplex. 
\end{theorem}

Inside a polytopal complex $C$, a \emph{free} face $\sigma$ is a face strictly contained in only one other face of $C$. An \emph{elementary collapse} is the deletion of a free face $\sigma$ from a polytopal complex~$C$. We say that $C$ \emph{(elementarily) collapses} onto $C-\sigma$, and write $C\searrow_e C-\sigma.$ We also say that the complex $C$ \emph{collapses} to a subcomplex $C'$, and write~$C\searrow C'$, if $C$ can be reduced to $C'$ by a sequence of elementary collapses. A \emph{collapsible} complex is a complex that collapses onto a single vertex. Collapsibility is, clearly, a combinatorial property (i.e. it only depends on the combinatorial type), and does not depend on the geometric realization of a polytopal complex. The connection to discrete Morse theory is highlighted by the following simple result:

\begin{theorem}[Forman \cite{FormanADV}]\label{lem:equivalence}
Let $C$ be a polytopal complex. The complex $C$ is collapsible if and only if $C$ admits a discrete Morse function with only one critical face.
\end{theorem}

Collapsible complexes are contractible; collapsible PL manifolds are necessarily balls \cite{Whitehead}.
The following facts are easy to verify.

\begin{lemma}\label{lem:ccoll}
Let $C$ be a simplicial complex, and let $C'$ be a subcomplex of $C$.  Then the cone over base $C$ collapses to the cone over $C'$.
\end{lemma}

\begin{lemma}\label{lem:cecoll}
Let $v$ be any vertex of any simplicial complex $C$. If $\Lk(v,C)$ collapses to some subcomplex $S$, then $C$ collapses to 
$(C-v)\cup (v\ast S).$ In particular, if $\Lk(v,C)$ is collapsible, then  $C\searrow C-v$.  
\end{lemma}

\begin{lemma}\label{lem:uc}
Let $C$ denote a simplicial complex that collapses to a subcomplex $C'$. Let $D$ be a simplicial complex such that $D\cup C$ is a simplicial complex. If $D\cap C=C'$, then $D\cup C\searrow D$.
\end{lemma}

\begin{proof}
It is enough to consider the case $C\searrow_e C'=C-\sigma$, where $\sigma$ is a free face of $C$. For this, notice that the natural embedding $C\mapsto D\cup C$ takes the free face $\sigma\in C$ to a free face of $D\cup C$.
\end{proof}

\emph{Non-evasiveness} is a further strengthening of collapsibility that emerged in theoretical computer science~\cite{KahnSaksSturtevant}. A $0$-dimensional complex is \emph{non-evasive} if and only if it is a point. Recursively, a $d$-dimensional simplicial complex ($d>0$) is \emph{non-evasive} if and only if there is some vertex $v$ of the complex whose link and deletion are both non-evasive. Again, non-evasiveness is a combinatorial property only.

The notion of non-evasiveness is rather similar to vertex-decomposability, a notion defined only for \emph{pure} simplicial complexes \cite{ProvanBillera}; to avoid confusions, we recall the definition and explain the difference in the lines below. A $0$-dimensional complex is \emph{vertex-decomposable} if and only if it is a finite set of points. In particular, not all vertex-decomposable complexes are contractible. Recursively, a $d$-dimensional simplicial complex ($d>0$) is \emph{vertex-decomposable} if and only if it is pure and there is some vertex $v$ of the complex whose link and deletion are both vertex-decomposable. Note that in particular the link and the deletion in a vertex decomposition have to be pure. All vertex-decomposable contractible complexes are easily seen to be non-evasive, similar to the basic fact that contractible shellable complexes are collapsible \cite{BTM}. An important difference arises when considering cones. It is easy to see that the cone over a simplicial complex $C$ is vertex-decomposable if and only if $C$ is. In contrast,

\begin{lemma}[cf.~Welker~\cite{Welker}] \label{lem:conev}
The cone over any simplicial complex is non-evasive.
\end{lemma}

By Lemma~\ref{lem:cecoll} every non-evasive complex is collapsible. As a partial converse, we also have the following lemma
\begin{lemma}[cf.~Welker~\cite{Welker}]
The derived subdivision of every collapsible complex is non-evasive. 
In particular, the derived subdivision of any non-evasive complex is non-evasive.
\end{lemma}

A \emph{non-evasiveness step} is the deletion from a simplicial complex $C$ of a single vertex whose link is non-evasive. Given two simplicial complexes $C$ and $C'$, we write $C\searrow_{\NE} C'$ if there is a sequence of non-evasiveness steps which lead from $C$ to $C'$. We will need the following lemmas, which are well known and easy to prove: 

\begin{lemma}\label{lem:nonev}
If $C\searrow_{\NE} C'$, then $\mathrm{sd}^m  C   \searrow_{\NE}   \mathrm{sd}^m  C'$ for all non-negative $m$.
\end{lemma}

\begin{lemma}\label{lem:cone}
Let $v$ be any vertex of any simplicial complex $C$. Let $m \ge 0$ be an integer. Then $(\mathrm{sd}^m   C)-v   \searrow_{\NE}   \mathrm{sd}^m (C-v)$.
In particular, if $\mathrm{sd}^m  \Lk(v,C)$ is non-evasive, then $\mathrm{sd}^m   C   \searrow_{\NE}   \mathrm{sd}^m (C-v)$.
\end{lemma}

\begin{proof}
The case $m=0$ is trivial. We treat the case $m =1$ as follows: The vertices of $\sd C$ correspond to faces of~$C$, the vertices that have to be removed in order to deform $(\sd C) -v$ to $\sd (C-v)$ correspond to the faces of $C$ strictly containing $v$. The order in which we remove the vertices of $(\sd C)-v$ is by increasing dimension of the associated face. Let $\tau$ be a face of $C$ strictly containing $v$, and let $w$ denote the vertex of $\sd C$ corresponding to $\tau$. Assume all vertices corresponding to faces of $\tau$ have been removed from $(\sd C)-v$ already, and call the remaining complex~$D$. Denote by $\mathrm{L}(\tau,C)$ the set of faces of $C$ strictly containing $\tau$, and let $\F(\tau-v)$ denote the set of nonempty faces of $\tau-v$.
Then $\Lk(w, D)$ is combinatorially equivalent to the order complex of $\mathrm{L}(\tau,C)\cup \F(\tau-v)$, whose elements are ordered by inclusion. 
 Every maximal chain contains the face $\tau-v$, so $\Lk(w, D)$ is a cone, which is non-evasive by Lemma~\ref{lem:conev}. Thus, we have $D \searrow_{\NE} D-w$. The iteration of this procedure shows $(\sd C) -v \searrow_{\NE}  \sd (C-v)$, as desired.
The general case follows by induction on $m$: Assume that $m\geq 2$. Then 
\begin{align*}
(\mathrm{sd}^{m} C)-v\ =\  (\sd (\mathrm{sd}^{m-1} C))-v \  & \searrow_{\NE} \  \sd ((\mathrm{sd}^{m-1} C) - v)\\ & \searrow_{\NE}\  \sd (\mathrm{sd}^{m-1} (C - v))\ =\  \mathrm{sd}^m (C - v), 
\end{align*}
by applying the inductive assumption twice, and Lemma~\ref{lem:nonev} for the second deformation.
\end{proof}

\subsection{Acute and non-obtuse triangulations} \label{subsec:acute}

A simplex is called \emph{acute} (resp.\ \emph{non-obtuse}) if the dihedral angle between any two facets is smaller than $\frac{\pi}{2}$ (resp.\ smaller or equal than $\frac{\pi}{2}$). In any acute simplex, all faces are themselves acute simplices. In particular, all triangles in an acute simplex are acute in the classical sense. The same holds for non-obtuse simplices. A simplex is called \emph{equilateral} or \emph{regular} if all edges have the same length. Obviously, all equilateral simplices are acute and all acute simplices are non-obtuse. 
The next, straightforward lemma characterizes these notions in terms of orthogonal projections.

\begin{lemma}
A $d$-simplex $\Delta$ is acute (resp.\ non-obtuse) if and only if, for each facet $F$ of $\partial \Delta$, the closest-point projection $\pi_{\Delta - F}\, \mathrm{span} F$ of the vertex $\Delta - F$ to the affine span of $F$ intersects the relative interior of $F$ (resp.\ intersects $F$). 
\end{lemma}

In 2004, Eppstein--Sullivan--\"{U}ng\"{o}r showed that $\mathbb{R}^3$ can be tiled into acute tetrahedra~\cite{EppsteinEtAl}. This was strengthened by Van der Zee et al.~\cite{VanderZee} and Kopczy\'{n}ski--Pak--Przytycki~\cite{KopczynskiPakPrzytycki}, who proved that the unit cube in $\mathbb{R}^3$ can be tiled into acute tetrahedra. In contrast, there is no geometric triangulation of the $4$-cube into acute $4$-simplices~\cite{KopczynskiPakPrzytycki}.
For $d \ge 5$, neither $\mathbb{R}^d$ nor the $(d+1)$-cube have acute triangulations, cf.~\cite{KopczynskiPakPrzytycki}. 
In contrast, by subdividing a cubical grid, one can obtain non-obtuse triangulations of $\mathbb{R}^d$ and of the $d$-cube for any $d$. So, acute is a much more restrictive condition than non-obtuse.

\begin{figure}[htb]
	\centering
  \includegraphics[width=.94\linewidth]{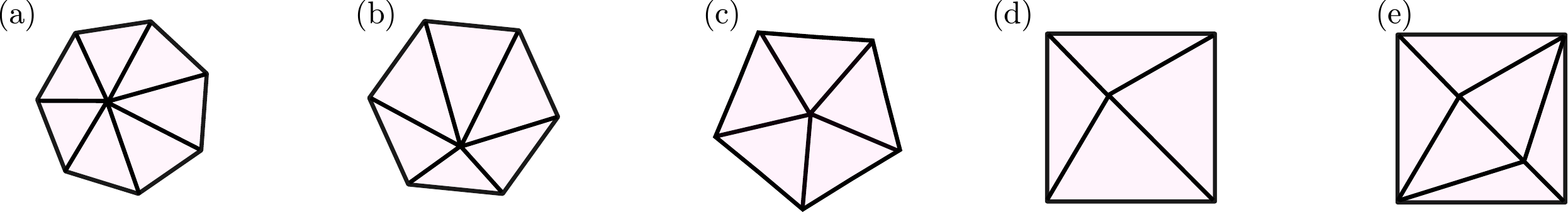}
	\caption{\footnotesize Complex { (a)} is 
$\CAT(-1)$ with an equilateral hyperbolic metric. Complex {(b)} is $CAT(0)$ with the equilateral flat metric. Complex {(c)} is not $\CAT(0)$ with the equilateral flat metric; it is $\CAT(0)$ with the acute piecewise euclidean metric that assigns length $1$ resp.\ $1.4$ to all interior resp.\ boundary edges. Complex {(d)} is flat with the non-obtuse piecewise euclidean metric that assigns length $1$ resp.\ $\sqrt{2}$ to all interior resp.\ boundary edges. Finally, only an obtuse metric can make complex {(e)} $\CAT(0)$; but the complex is $\CAT(1)$ with a non-obtuse metric.}
	\label{fig:pentagon}
\end{figure}

Let $C$ be an intrinsic simplicial complex in which every face of $C$ is isometric to a regular euclidean simplex. If such $C$ is $\CAT(0)$, we say that it is \emph{$\CAT(0)$ with the equilateral flat metric}. More generally, a $\CAT(k)$ intrinsic simplicial complex $C$ is \emph{ $\CAT(k)$ with an acute metric} (resp.\ \emph{$\CAT(k)$ with a non-obtuse metric}) if every face of $C$ is acute (resp.\ non-obtuse), see Figure~\ref{fig:pentagon}.

\section{Collapsibility and curvature bounded above} \label{sec:2}

\subsection{Gradient Matchings and Star-Minimal Functions}\label{subsec:1}
Here we obtain non-trivial Morse matchings on a simplicial (or polytopal) complex, by studying real-valued continuous functions on the complex.

\begin{definition}[Star-minimal]
Let $C$ be an intrinsic simplicial complex. A (nonlinear) function $f : |C| \rightarrow \mathbb{R}$ is called \emph{star-minimal} if it satisfies the following three conditions:
\begin{compactenum}[(i)]
\item $f$ is continuous,
\item on the star of each face of $C$, the function $f$ has a unique absolute minimum, and
\item no two vertices have the same value under $f$.
\end{compactenum}
\end{definition}

Condition (iii) is just a technical detail, since it can be forced by `wiggling' the complex a bit, or by perturbing $f$. Alternatively, we can perform a `virtual wiggling' by choosing a strict total order $\triangleleft$ on the vertices of $C$ such that so that if $f(x) < f(y)$, then $x\triangleleft y$. 

Conditions (i) and (ii) are also not very restrictive: Every generic continuous function on any simplicial complex is star-minimal.

Our next goal is to show that any star-minimal function on a complex $C$ naturally induces a certain type of Morse matching, called \emph{gradient matching}. The key is to define a ``pointer function'' $y_f: C \rightarrow C$, which intuitively maps each face into the ``best'' vertex of its star. (How good a vertex is, is decided simply by looking at its image under $f$.) Unlike $f$, which is defined on $|C|$, the map $y_f$ is purely combinatorial. Later, we will obtain a matching from $y_f$ basically by pairing every still unmatched face $\sigma$ together with the face $\sigma \ast y_f (\sigma)$.  

\begin{figure}[ht]
	\centering
  \includegraphics[width=.17\linewidth]{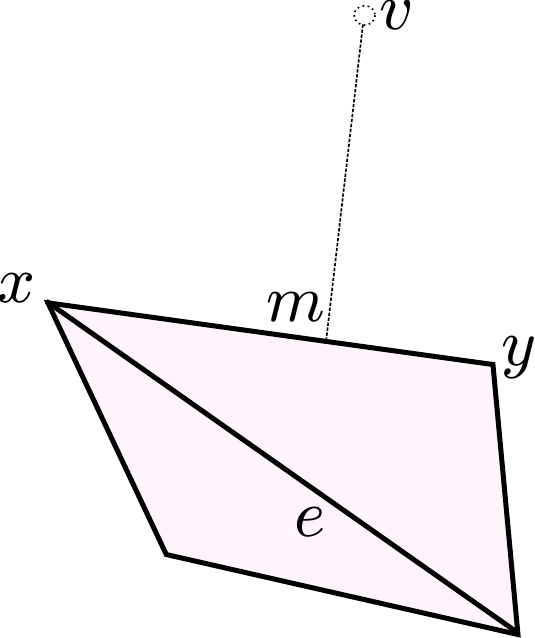}
 	\caption{\footnotesize How to match an edge $e$ with respect to the function $f = $ distance from some point $v$. This (nonlinear) function has a unique minimum $m$ on $|\St (e,C)|$. The inclusion-minimal face containing $m$ in its interior is spanned by the vertices $x$ and $y$. Between these two vertices, we choose the one with minimal distance from $v$. 
If it is~$x$, since $x$~is contained in~$e$, we do nothing. If it is~$y$, which does not belong to $e$, we match $e$ with $e \ast y$.}
	\label{fig:matching}
\end{figure}

In detail: Let $\sigma$ be a face of $C$. On the star of $\sigma$, the function $f$ has a unique minimum $m$. We denote by $\mu (\sigma)$ the inclusion-minimal face among all the faces of $\St(\sigma,C)$ that contain the point $m$. We set 
\[y_f (\sigma) \:= \ y, \]
where $y$ is the $f$-minimal vertex of $\mu (\sigma)$.
Since $f$ is injective on the vertices of $C$, the pointer function $y_f$ is well-defined. 

Next, we define a matching $\Theta_f : C \longrightarrow C$ associated with the function $f$. The definition is recursive on the dimension: In other words, we start defining $\Theta_f$ for the faces of lower dimension, working all the way up to the facets. Let $\sigma$ be a face of $C$. Set $\Theta(\emptyset) := \emptyset$. If for all faces $\tau \subsetneq \sigma$ one has $\Theta_f (\tau) \ne \sigma$, we define 
\[ \Theta_f(\sigma) \ := \ y_f (\sigma) \ast \sigma.\]
Note that every face of $C$ is either in the image of $\Theta_f$, or in its domain, or both. In the latter case, we have $\Theta_f (\sigma) = \sigma$. 

\begin{definition}[Gradient matching]
Let $\Theta : C \longrightarrow C$ be (the partial function associated to) a matching on the faces of a simplicial complex $C$. We say that $\Theta$ is a \emph{gradient matching} if 
$\Theta 
 =  
\Theta_f$
for some star-minimal function $f : |C| \rightarrow \mathbb{R}$.
\end{definition}

Next we show that all gradient matchings are indeed Morse matchings.
  
\begin{theorem}\label{thm:formula}
Let $C$ be a simplicial complex. Let $f:|C|\rightarrow \mathbb{R}$ be any star-minimal continuous function on the underlying space of~$C$. Then the induced gradient matching $\Theta_f$ is a Morse matching. Moreover, the map 
\[\sigma \mapsto (y_f(\sigma),\sigma)\]
yields a bijection between the set $\mathfrak{C}$ of the (nonempty) critical faces and the set \[\mathfrak{P} := \left \{
(v,\tau) \; | \; v \subset \tau, \,\ y_f(\tau)= v \textrm{ and } y_f(\tau- v)\ne v \right \}.\] 
In particular, the complex $C$ admits a discrete Morse function with $c_i$ critical $i$-simplices, where  
\[ c_i = \# \, \left\{
(v,\tau) \; | \; 
v \subset \tau, \, \dim \tau =i, \, y_f(\tau)= v  \textrm{ and } 
y_f(\tau- v )\ne v \right\}.
\]
\end{theorem}

\begin{proof}
Our proof has four parts:
\begin{compactenum}[\bf (I)]
\item the pairs $(\tau,\Theta_f(\tau))$ form a discrete vector field;
\item this discrete vector field contains no closed gradient path;
\item the image of $\sigma \mapsto (y_f(\sigma),\sigma)$ is in $\mathfrak{P}$.
\item the map $\sigma \mapsto (y_f(\sigma),\sigma)$ is a bijection.
\end{compactenum}

Parts (I) and (II) guarantee that the pairs $(\tau,\Theta_f(\tau))$ form a Morse matching; parts (III) and (IV) determine how many of the faces are critical. The latter count will turn out to be crucial for establishing Theorem~\ref{thm:dischadamard} below.
 
\par \medskip \noindent \textbf{Part (I).} \emph{The pairs $(\tau,\Theta_f(\tau))$ form a discrete vector field.} 
We have to show that $\Theta_f$ is injective. Suppose $\Theta_f$ maps two distinct $k$-faces $\sigma_1,\sigma_2$ to the same $(k+1)$-face $\Sigma$. By definition of $\Theta_f$,
\[y_f(\sigma_1)\ast \sigma_1 = \Sigma = y_f(\sigma_2)\ast \sigma_2.\]
All vertices of $\Sigma$ are contained either in $\sigma_1$ or in $\sigma_2$. So $\St (\sigma_1,C) \cap \St (\sigma_2,C) =\St (\Sigma,C)$. Moreover, each $y_f(\sigma_i)$ belongs to $\Sigma$ but not to $\sigma_i$; so $y_f (\sigma_i)$ must belong to $\sigma_{3-i}$. This means that the function $f$ attains its minimum on $\St (\sigma_i,C)$ at a point $m_i$ which lies in $\sigma_{3-i}$. Since on $\St (\sigma_1,C) \cap \St (\sigma_2,C) =\St (\Sigma,C)$, the function $f$ has a unique minimum, we obtain $m_1 = m_2$. But then $y_f(\sigma_1) = y_f(\sigma_2) = y_f(\Sigma).$ 
Since $y_f(\sigma_1)\ast \sigma_1 =y_f(\sigma_2)\ast \sigma_2$, it follows that $\sigma_1 = \sigma_2$, a contradiction.

\par \medskip \noindent \textbf{Part (II).} \emph{The discrete vector field contains no closed gradient path.} 
Consider the function $\min_{\St (\sigma,C))} f$ from $C$ to $\R$. Choose an arbitrary gradient path $(\sigma_1,\Sigma_1), \ldots, (\sigma_k,\Sigma_k)$.  In this path, consider a pair $(\sigma_i, \Sigma_i)$. By definition of gradient path, the faces $\sigma_i$ and $\Sigma_i$ are matched with one another, and $\sigma_{i+1}$ is matched with $\Sigma_{i+1}$ and so on. In particular $\Theta_f (\sigma_{i}) = \Sigma_i$ and $\Theta_f (\sigma_{i+1}) = \Sigma_{i+1}$, and hence $\min_{\St (\sigma_i,C)} f= \min_{\St (\Sigma_i,C)} f$. We claim that in this case
\[\min_{\St (\Sigma_i,C)} f> \min_{\St (\sigma_{i+1},C)} f.\]
In fact, suppose by contradiction that $\min_{\St (\Sigma_i,C)} f= \min_{\St (\sigma_{i+1},C)} f$. Note that $\sigma_{i+1} \subset \Sigma_i$. By the star-minimality of~$f$, the local minimum of $f$ must be unique. So 
\[y_f (\sigma_{i}) = y_f (\Sigma_i) = y_f (\sigma_{i+1}).\]
So $y_f (\sigma_{i+1})$ is the unique vertex in $\Sigma_i$ but not in $\sigma_{i}$. In particular, we have that $y_f (\sigma_{i+1})$ belongs to $\sigma_{i+1}$. Yet this contradicts the fact that $\sigma_{i+1}$ is matched with coface $\Sigma_{i+1}$.
In conclusion, for any gradient path $\sigma_1, \ldots, \sigma_m$ one has $\min_{\St (\sigma_i,C)} f=\min_{\St (\Sigma_i,C)} f> \min_{\St (\sigma_{i+1},C)} f$. This excludes the possibility that the gradient path is closed.

\par \medskip \noindent \textbf{Part (III).} \emph{The image of $\sigma \mapsto (y_f(\sigma),\sigma)$ is in $\mathfrak{P}$.} Recall that 
\[\mathfrak{P} := \left \{
(v,\tau)  :  v \subset \tau, \,\ y_f(\tau)= v \textrm{ and } y_f(\tau- v)\ne v \right \}.\] 
Consider an arbitrary vertex $w$ of $C$. Either $y_f(w) \ne w$, or $y_f(w)=w$. If $y_f(w)$ is a vertex $x$ different from $w$, then $w$ is in the domain of $\Theta_f$; moreover, we see that $\Theta_f (w)$ is the edge $[x,w]$ and $w$ is not critical. If $y_f(w) = w$, again $w$ is in the domain of $\Theta_f$; one has $\Theta_f (w) = w$, so $w$ is critical. In this case it is easy to verify that $(w,w)$ belongs to $\mathfrak{P}$.
 
Now, consider a critical face $\sigma$ of dimension $k\ge 1$. Since $\sigma$ is critical, there is no $(k-1)$-face $\tau$ such that $\Theta_f (\tau) = \sigma$. By the definition of $\Theta_f$,  the face $\sigma$ is in the domain of $\Theta_f$; so since $\sigma$ is critical we must have $\Theta_f (\sigma) = \sigma$. So the vertex $y_f(\sigma)$ belongs to $\sigma$. In order to conclude that $(y_f(\sigma),\sigma)\in \mathfrak{P}$, we only have to prove that $y_f(\sigma- y_f(\sigma))\neq y_f(\sigma)$. Let us adopt the abbreviations $v:=y_f(\sigma)$ and $\delta:=\sigma- v$. Suppose by contradiction that $y_f(\delta)= v$. If $\delta$ were not in the image of any of its facets under the map $\Theta_f$, then we would have $\Theta_f(\delta)=\sigma$, which would contradict the assumption that $\sigma$ is a critical face. So, there has to be a codimension-one face $\rho$ of $\delta$ such that $\Theta_f(\rho)= \delta$. In other words, we see that $y_f(\rho)$ is the unique vertex that belongs to $\delta$ but not to $\rho$. Clearly $\am_{\, \St (\rho,C)\,} f \in \St (y_f(\rho), C)$, whence
\[\am_{\, \St (\rho,C) } \, f 
\, = \, 
\am_{\, \St (\rho \ast  y_f(\rho),C) } \, f 
\, = \, 
\am_{\, \St (\delta,C) } \, f.\]

\noindent So $y_f(\rho)=y_f(\delta)$. Recall that we are assuming $y_f (\delta) = v$, where $\delta := \sigma- v$. Hence
\[ y_f(\rho)=y_f(\delta) = y_f (\sigma- v) = v \; \notin \, \sigma- v.\]
This contradicts the fact that $\Theta_f(\rho)=\sigma- v$. Thus, the assumption $y_f(\delta)= v$ must be wrong. 

\par \medskip \noindent \textbf{Part (IV).} \emph{The map $\sigma \mapsto (y_f(\sigma),\sigma)$ from $\mathfrak{C}$ to $\mathfrak{P}$ is a bijection.}

The map $s:\sigma \mapsto (y_f(\sigma),\sigma)$ is clearly injective. Let us verify surjectivity. Consider a pair$(v,\tau)$ in $\mathfrak{P}$. By definition of $s$,
$y_f(\tau)= v $ and $y_f(\tau- v)\ne v$. Assume that $s(\tau)\notin
\mathfrak{P}$, or, equivalently, that $\tau$ is not critical.

Since $y_f(\tau)= v \subset \tau$, there must be a facet $\eta$ of $\tau$
such that $\Theta_f (\eta)=\tau$. By definition of $\Theta_f$, we have $\am_{\St
(\eta,C)}\in \St (\tau,C)$ and hence $\am_{\St (\eta,C)}= \am_{\St (\tau,C)}$.
Since $y_f$ only depends on the point where $f$ attains the minimum,
$y_f(\tau)=y_f(\eta)= y_f(\tau - v) = v $. This contradicts the
assumption that $(v,\tau)$ is in $\mathfrak{P}$. Thus $\tau$ is
critical, as desired.
\end{proof}

\begin{cor}\label{cor:c} 
Every gradient matching is a Morse matching.
\end{cor}

In fact, one can characterize gradient matchings within Morse matchings as follows: A Morse matching is a gradient matching if and only if for every matching pair $(\sigma,\Theta(\sigma))$ there is a facet $\Sigma\supset \Theta(\sigma)$ of $C$ such that, for any face $\tau$ of $\Sigma$ that contains $\sigma$ but does not contain $\Theta(\sigma)$, we have 
$\Theta (\tau)=\tau\ast \Theta(\sigma).$

\subsection{Discrete Hadamard--Cartan Theorem} \label{subsec:2}
In this section, we prove our discrete version of the Hadamard--Cartan theorem,
namely, that CAT(0) complexes with convex vertex stars are collapsible ({\bf Theorem~\ref{thm:dischadamard}}). Our theorem strengthens a result by Crowley, whose proof technique we shall briefly discuss. 

Crowley's approach uses the classical notion of \emph{minimal disk}, which is the disk of minimum area spanned by a $1$-sphere in a simply connected complex. Gromov \cite{GromovHG} and Gersten \cite{Gersten, Gersten2} have studied minimal disks in connection with group presentations and the word problem; later Chepoi and others~\cite{BandeltChepoi, ChepoiOsaj} have used them to relate $\CAT(0)$ cube complexes to median graphs, and systolic complexes to bridged graphs. To collapse a complex onto a fixed vertex $v$, Crowley's idea is to reduce the problem to the two-dimensional case, by studying the minimal disk(s) spanned by two geodesics that converge to $v$. Her argument is based on two observations:
\begin{compactenum}[(1)]
\item If the complex is three-dimensional, these minimal disks are $\CAT(0)$ with the equilateral flat metric, as long as the starting complex is. This was proven in~\cite[Theorem 2]{Crowley}. A similar result was obtained for ``systolic'' complexes in \cite{ChCAT} and \cite{JanusSwiat}, and for ``weakly systolic'' complexes in~\cite[Claim 1]{ChepoiOsaj}; \cite[Theorem 2]{Crowley} implies that all three-dimensional simplicial complex with equilateral simplices are systolic.
\item Any contractible $3$-dimensional pseudomanifold is collapsible if and only if it admits a discrete Morse function without critical edges. 
\end{compactenum}
Neither of these two facts extends to higher-dimensional manifolds:
\begin{compactenum}[(1)]
\item If we take the join of a triangle with a cycle of length $5$, the resulting complex is obviously collapsible, but not with Crowley's argument (even though it does support a $\CAT(0)$ equilateral flat metric): The minimal disk bounded by the $5$-cycle contains a degree-five vertex, and is consequently not $\CAT(0)$ with the equilateral flat metric. 
\item For any integers $m \ge 0$ and $d \ge 6$, some contractible $d$-manifold different from a ball admits a discrete Morse function with $1$ critical vertex, $m$ critical $(d-3)$-faces, $m$ critical $(d-2)$-faces, and no further critical face. (In particular, no critical edges.) This follows from discretizing, as in \cite{Gallais10} and \cite{B-Smoothing}, the smooth result by Sharko~\cite[pp.~27--28]{Sharko}. 
\end{compactenum}

\noindent Our new approach consists in applying the ideas of Section~\ref{subsec:1} to the case where $C$ is a $\CAT(0)$ complex, and the function $f : C \rightarrow \mathbb{R}$ is the distance from a given vertex $v$ of $C$. Because of the $\CAT(0)$ property, this function is star-minimal --- so it will induce a convenient gradient matching on the complex. The main advantage of this new approach is that it works in all dimensions.

\begin{theorem}\label{thm:dischadamard}
Let $C$ be a  $\CAT(k)$ intrinsic simplicial complex such that
\begin{compactenum}[\rm (i)]
\item for each face $\sigma$ in $C$, the underlying space of $\St (\sigma,C)$ in $C$ is convex, and
\item if $k >0$, every facet is contained in some ball of radius $< \tfrac{1}{2}\pi k^{\nicefrac{-1}{2}}$.
\end{compactenum}
Then $C$ is collapsible.
\end{theorem}
\vskip -3mm
\enlargethispage{3mm}
\begin{proof}
Fix $x$ in $C$ such that, if $k >0$, then every facet is contained in some ball of radius $<\tfrac{1}{2}\pi k^{\nicefrac{-1}{2}}$ around $x$.  Let $\mathrm{d} : C \longmapsto \mathbb{R}$ be the distance from $x$ in $C$, and let $w$ denote the vertex of $C$ that minimizes $\mathrm{d}$. 
Let us perform on the face poset of $C$ the Morse matching constructed in Theorem~\ref{thm:formula}. The vertex $w$ will be mapped onto itself and for each vertex $u \neq w$ we have~$y_{\mathrm{d}}(u) \ne u$. So  every vertex is matched with an edge, apart from $w$, which is the only critical vertex.

By contradiction, suppose there is a critical face $\tau$ of dimension $\ge 1$. Set $v:=y_{\mathrm{d}}(\tau)$. By Lemma~\ref{lem:unique}, on $\St (\tau,C)$ the function~${\mathrm{d}}$ attains its minimal value in the relative interior of a face $\sigma_v \in \St (\tau,C)$ that contains~$v$.
Let $\delta$ be any face of $\St (\tau- v,C)$ containing $\sigma_v$. Clearly $\delta$ contains~$v$. Thus $\St (\tau,C)$ and $\St (\tau- v,C)$ coincide in a neighborhood of~$v$.
By Lemma~\ref{lem:unique}, we have $\am_{\, \St (\tau- v,C)} \, {\mathrm{d}}= \am_{\, \St(\tau,C)} \, {\mathrm{d}}$.
Therefore $y_{\mathrm{d}}(\tau- v)=y_{\mathrm{d}}(\tau)= v$. This means that 
\[(v,\tau) \notin \{(v,\tau): v \subset \tau, \, y_{\mathrm{d}}(\tau)=~v \textrm{ and } y_{\mathrm{d}}(\tau- v) \ne v\};\] 
hence by Theorem~\ref{thm:formula} $\tau$ is not critical, a contradiction.
\end{proof}

\begin{rem} \label{rem:localconvexity}
In a $\CAT(0)$ space, any closed connected locally convex subset is also convex \cite{BW}. So in the assumptions of Theorem~\ref{thm:dischadamard} it suffices to check condition (i) locally: We can replace condition (i) with the request that $\St (\sigma,C)$ is convex for every ridge $\sigma$ (ridges are facets of facets of a simplicial complex).
In fact, if all ridges have convex stars, the closest-point projection to $\St(v,C)$ is a well-defined map and it is locally non-expansive for every vertex~$v$. 
But then the closest-point projection to $\St(v,C)$ is also globally non-expansive. Now let $x$ and $y$ be any two points in $\St(v,C)$, and let us assume the geodesic $\gamma$ from $x$ to $y$ leaves $\St(v,C)$. Let us project $\gamma$ to $\St(v,C)$. The result is a curve $\gamma'$ connecting $x$ and $y$ that is obviously lying in $\St(v,C)$, and not longer than $\gamma$. This contradicts the uniqueness of segments in $\CAT(0)$ spaces.
\end{rem}

\begin{cor} \label{cor:CrowleyA}
Let $C$ be an intrinsic simplicial complex. Suppose that $C$ is $\CAT(0)$ with a non-obtuse metric. 
Then $C$ is collapsible.
\end{cor}

\begin{proof}
By the assumption, there is a metric structure on $C$ of non-positive curvature, such that every face of $C$ is non-obtuse. In non-obtuse triangulations, the star of every ridge is convex. In fact, let $\Sigma,\Sigma'$ be two facets containing a common ridge $R$. Since $\Sigma$ and $\Sigma'$ are convex, and their union is locally convex in a neighborhood of $R$, we have that $\Sigma \cup \Sigma'$ is locally convex. Since the embedding space is $\CAT(0)$, convexity follows as in Remark~\ref{rem:localconvexity}.
\end{proof}

\begin{cor} \label{cor:CrowleyB}
Every intrinsic simplicial complex that is $\CAT(0)$ with the equilateral flat metric is collapsible.
\end{cor}

\begin{cor}[Crowley~{\cite{Crowley}}] \label{cor:CrowleyC}
Every $3$-pseudomanifold that is $\CAT(0)$ with the equilateral flat metric is collapsible.
\end{cor}

\begin{rem} Technical details aside, our proof is based on the idea of ``distance from a basepoint''. This idea has widely been used in mathematics, and in particular in graph theory, where it lies at the core of several algorithms, such as Bread-First-Search. Using a refinement of Bread-First-Search, Chepoi and Osajda proved that weakly systolic complexes are dismantlable, and in particular collapsible. Their result can thus be viewed as a ``more graph-theoretic'' version of Theorem \ref{thm:dischadamard}.
\end{rem}

\subsubsection*{Extension to polytopal complexes}
Theorem~\ref{thm:dischadamard} can be extended even to polytopal complexes that are not simplicial. The key for this is Bruggesser--Mani's rocket shelling of polytope boundaries, cf. e.g. Ziegler~\cite[Sec.~8.2]{Z}. The following, well-known lemma follows from the fact that there is a (rocket) shelling of $\partial P$ in which $\St(\sigma, \partial P)$ is shelled first~\cite[Cor.~8.13]{Z}:

\begin{lemma}[Bruggesser--Mani]\label{lem:BruMani}
For any polytope $P$ and for any face $\mu$,  $P$ collapses onto 
$\St(\mu, \partial P)$.
\end{lemma}


\begin{theorem}\label{thm:dischadamardpoly}
Let $C$ be any  $\CAT(k)$ (intrinsic) polytopal complex such that
\begin{compactenum}[\rm (i)]
\item for each face $\sigma$ in $C$, the underlying space of $\St (\sigma,C)$ in $C$ is convex, and
\item if $k >0$,  every facet is contained in some ball of radius $< \tfrac{1}{2}\pi k^{\nicefrac{-1}{2}}$.
\end{compactenum}
Then $C$ is collapsible.
\end{theorem}

\begin{proof}
We make the stronger claim that every polytopal complex that admits a function $f:|C|\rightarrow \R$ that takes a unique local minimum on each vertex star is collapsible. This situation clearly applies: Fix $x$ in $C$ such that, if $k >0$, then every facet is contained in some ball of radius $<\tfrac{1}{2}\pi k^{\nicefrac{-1}{2}}$ around $x$.  Let $\mathrm{d} : C \longmapsto \mathbb{R}$ be the distance from $x$ in $C$. This $\mathrm{d}$ is a function that has a unique local minimum on the star of each face by Lemma~\ref{lem:unique}.

We prove the claim by induction. Let $\sigma$ be a facet of $C$ maximizing $\min_{\sigma} f$, and let $\mu$ denote the strict face of $\sigma$ that minimizes  $f$. Let $F \subset \St (\mu,C)$ be the subcomplex induced by the facets of $C$ that attain their minimum at $\mu$. 
By Lemma~\ref{lem:BruMani}, we can collapse each facet $P$ of $F$ to $\St(\mu, \partial P)$. Hence, we can collapse $F$ to 
\[ \bigcup_{P \in F} \St(\mu, \partial P)  = \St (\mu,C) \: \cap \: \bigcup_{P \in F} \partial P, \]
where $P$ ranges over the facets of $F$. In particular, we can collapse $C$ to $C':=C-F$.

\begin{figure}[ht]
	\centering
	\vskip-1mm
  \includegraphics[width=.26\linewidth]{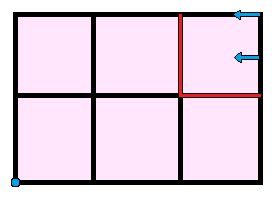}
 	\caption{\footnotesize How to collapse this $\CAT(0)$ cubical complex to the bottom left vertex $v$. We proceed by induction: Here $\sigma$ is the top right square, and $F$ is the complex highlighted in red.}
	\label{fig:grid}
\end{figure}

It remains to show the existence of a function on $C'$ that attains a unique local minimum on each star. Let us show that the restriction $f_{|C'}$ of $f$ to $C'$ is the function we are looking for. Assume that for some vertex~$w$ of~$C$ the function $f_{|C'}$ attains two local minima on $\St(w,C')$. Clearly $w$ is in $F$. Let $x$ be the absolute minimum of~$f$ restricted to~$\St (w,C')$; let $y$ be the other (local) minimum. Let $P$ be a facet of $C-F$ containing $x$. When restricted to $C$, the function $f$ attains a unique local minimum on the star of every face. Therefore, the point $y$ must lie in $F$. In particular, the facet $P$ must contain the minimum of $f$ on $F$. But $y$ is not that local minimum since $P$ is not in $F$, so $f$ takes two local minima on $P$, in contradiction with the assumption on~$f$.
\end{proof}

In particular, Theorem~\ref{thm:dischadamardpoly} holds for \emph{$\CAT(0)$ cube complexes}, which are complexes of regular unit cubes glued together to yield a $\CAT(0)$ metric space. These complexes have been extensively studied in the literature: See for example~\cite{Ardila}, \cite{BilleraHolmesVogtmann}, \cite{DavisJanus}, \cite{GromovHG}. 

\begin{cor} \label{cor:CubeComplexes}
Every  $\CAT(0)$ cube complex is collapsible.
\end{cor}

One instance of a $\CAT(0)$ cube complex is the space of phylogenetic trees, introduced by Billera, Holmes and Vogtmann~\cite{BilleraHolmesVogtmann}:

\begin{cor} \label{cor:phylo}
The space of phylogenetic trees is collapsible.
\end{cor}

\subsubsection*{Vertex-transitive triangulations} \label{subsec:evasiveness}
In this short section we include a connection of Theorem~\ref{thm:dischadamard} between metric geometry and the evasiveness conjecture, following a suggestion by Anders Bj\"{o}rner. A \emph{vertex-transitive} complex is a simplicial complex with $n$ vertices on  which the symmetric group acts transitively. An important open problem in theoretical computer science is whether there is any vertex-transitive non-evasive simplicial complex, apart from the simplex. This is known as \emph{evasiveness conjecture}~\cite{KahnSaksSturtevant}\cite{LutzVH}.

It is known that collapsibility is not enough to force a vertex-transitive complex to be a simplex~\cite{LutzVH}. 
However, non-evasiveness is strictly stronger than collapsibility. We have shown in Section~\ref{subsec:2} that the property of ``being $\CAT(0)$ with the equilateral flat metric'' is also strictly stronger than collapsibility. Thus it makes sense to compare it with vertex-transitivity, in parallel with the statement of the evasiveness conjecture. Here is a simple observation:

\begin{proposition} \label{prop:evasiveness}
Every vertex-transitive simplicial complex that is $\CAT(0)$ with the equilateral flat metric is a simplex.
\end{proposition}

This follows directly from the Bruhat-Tits fixed point theorem for $\CAT(0)$ spaces, cf.\ \cite{BT}. A direct proof is simple enough, though, so we present it here:

\begin{proof} Let $C$ be a vertex-transitive intrinsic simplicial complex and let $v_1, \ldots, v_n$ be the vertices of $C$. Let $a: |C| \longmapsto \mathbb{R}$ be the function $a(x) := \sum_{i=1}^n \mathrm{d}(x, v_i)$. When $C$ is $\CAT(0)$, the function $a$ is strongly convex and has therefore a \emph{unique} minimum $m$. 
Since the simplicial complex is vertex transitive, and the function $a$ is also invariant under the symmetries of $C$, we claim that 
\[ \mathrm{d}(m, v_1) = \mathrm{d}(m, v_2) = \; \ldots \; = \mathrm{d}(m, v_n),\]
so that $m$ minimizes simultaneously all functions $x \mapsto \mathrm{d}(x, v_i)$. In fact, were $\mathrm{d}(m, v_k) < \mathrm{d}(m, v_j)$ for some $k  \ne j$, we could find by symmetry a point $m'$ such that  \[m\neq m' \ \ \text{and} \ \ \sum_{i}\mathrm{d}(m', v_i) = \sum_{i}\mathrm{d}(m, v_i).\]
But then we would have $a(m) = a(m')$: A contradiction since $a$ has a unique minimum.

So, there is one point in the complex equidistant from all of the vertices. This implies that $C$ has only one facet, because for each facet $F$ of $C$, the unique point equidistant from all vertices of $F$ is the barycenter of $F$ itself. 
\end{proof}

\subsection{Collapsible, contractible and CAT(0) manifolds} \label{subsec:3}
We already mentioned the famous result obtained in 1939 by Whitehead:

\begin{theorem}[Whitehead] Let $M$ be a (compact) manifold with boundary. If some PL triangulation of~$M$ is collapsible, then $M$ is a ball.
\end{theorem}

Let us focus on non-PL triangulations. Let us introduce a convenient notation:
\begin{compactitem}[$\bullet$]
\item by {$\CAT_d$} we denote the class of $d$-manifolds homeomorphic to $\CAT(0)$ cube complexes;
\item by {$\mathrm{COLL}_d$} we denote the $d$-manifolds that admit a collapsible triangulation; 
\item by {$\mathrm{CONT}_d$} we denote all contractible $d$-manifolds.
\end{compactitem}
Moreover, by \emph{PL singular set of $M$} we mean the subcomplex given by the faces of $M$ whose link is not homeomorphic to a PL sphere or a PL ball.

\begin{theorem}\label{thm:ccc}
For each $d \ge 5$ one has
\[\CAT_d=\mathrm{COLL}_d = \mathrm{CONTR}_d,\]
whereas for $d=4$ one has
\[\CAT_4=\mathrm{COLL}_4\subsetneq \mathrm{CONTR}_4.\]
In particular, when $d \ge 5$ not all collapsible manifolds are balls.
\end{theorem}

\begin{proof} Clearly, every $d$-ball can be given a $\CAT(0)$ cubical structure (consisting of a single $d$-cube). By Theorem~\ref{thm:dischadamardpoly}, every $\CAT(0)$ cube complex is collapsible, and so is its first derived subdivision \cite{Welker}; hence $\CAT(0)$ cube complexes admit collapsible triangulations. Finally, collapsible complexes are contractible. This proves that 
\[\{\textrm{$d$-balls} \} \subset \CAT_d\subset \mathrm{COLL}_d \subset \mathrm{CONTR}_d \: \textrm{ for all } d.\] 

When $d \le 4$, every triangulation of a $d$-manifold is PL (This is non-trivial; for $d=4$, this statement relies on the Poincar\'e-Perelman theorem). By Whitehead's theorem, this implies $\mathrm{COLL}_d \subset \{\textrm{$d$-balls} \}$, so the first two containments above are actually equalities. All contractible $3$-manifolds are balls, so 
\[\{\textrm{$d$-balls} \} =  \CAT_d = \mathrm{COLL}_d =\mathrm{CONTR}_d \: \textrm{ for } d\le 3.\] 
In contrast, some contractible $4$-manifolds are not balls, like the Mazur manifold \cite{Mazur}; hence
\[\{\textrm{$4$-balls} \} = \CAT_4=\mathrm{COLL}_4 \subsetneq  \mathrm{CONTR}_4.\]

When $d \ge 5$, it is a classical result that 
$\{\textrm{$d$-balls} \} \subsetneq \mathrm{CONTR}_d$.
So all we need to show is that $\mathrm{CONTR}_d \subset \CAT_d$, namely, that every contractible $d$-manifold, $d\ge 5$,  admits a $\CAT(0)$ cube structure. 
By a result of Ancel and Guilbault \cite{AS}, every contractible manifold $M$ admits a triangulation $C$ such that:
\begin{compactenum}[(1)]
\item the PL singular set $S$ of $C$ lies in the interior of $C$,
\item the PL singular set $S$ of $C$ is a path (i.e.\ a graph homeomorphic to a curve),
\item $S$ is a deformation retract of $C$, and
\item $|C| {\setminus} |S|$ is homeomorphic to $M\times (0,1]$.
\end{compactenum}
Let us hyperbolize $C$ (as in Davis--Januskiewicz \cite{DavisJanus}), and pass to the universal cover. The resulting complex $C'$ is a $\CAT(0)$ cube complex. Since $S$ is a path, its image under hyperbolization and lift to universal cover is a disjoint union $S'$ of convex compact sets  in $C'$. Let $N$ denote the subcomplex of $C'$ consisting of faces that intersect a fixed connected component of $S'$. As in \cite[Thm.\ (5b.1)]{DavisJanus}, we have that $N$ is a cubical complex PL homeomorphic to $C$ and consequently homeomorphic to $M$. Furthermore $N$ is a convex subcomplex of a $\CAT(0)$ cube complex; in particular, it is itself a $\CAT(0)$ cube complex. In conclusion,
\[\{\textrm{$d$-balls} \} \subsetneq  \CAT_d = \mathrm{COLL}_d =\mathrm{CONTR}_d \: \textrm{ for } d \ge 5. \qedhere\] 
\end{proof}

\begin{rem}\label{rem:Funar}
The first author and Funar have recently extended this result to arbitrary complexes, proving that any collapsible simplicial complex is PL equivalent to a $\CAT(0)$  (and even $\CAT(-1)$) polyhedral complex \cite[Proposition 15]{AFunar}. With our Theorem~\ref{thm:dischadamardpoly}, this implies that a simplicial complex admits a collapsible subdivision if and only if it admits a $\CAT(0)$ cube structure. However, not all contractible simplicial complexes admit a collapsible subdivision: A well known counterexample is given by the Dunce Hat.
\end{rem}

Here is a curious consequences concerning discrete Morse theory:

\begin{cor}  \label{cor:Forman}
Forman's discrete Morse theory can give sharper upper bounds than smooth Morse Theory (or of PL handle theory) for the Betti numbers of a manifold.
\end{cor}

\begin{proof} Let $M$ be any collapsible (non-PL) triangulation of a contractible manifold different than a ball, as given by Theorem~\ref{thm:ccc}. By a classical result of Gleason, every open contractible manifold admits a smooth structure, and therefore a smooth handle decomposition; from that one can always obtain a PL triangulation and a PL handle decomposition. (See e.g.\cite{B-Smoothing} for the definitions.) However, since $M$ is not a ball, any of its (smooth or PL) handle decompositions must contain more than one handle. In particular the Betti vector $(1,0, \ldots, 0)$ is not a possible smooth or PL handle vector for $M$. Yet by Theorem~\ref{thm:ccc}, the same vector is a possible discrete Morse vector for (a suitable triangulation of) $M$, as long as $\dim M \ge 5$. 
\end{proof}

\begin{rem}
The situation changes if we restrict ourselves to PL triangulations. Indeed, every PL handle decomposition naturally yields a discrete Morse function, and vice versa, as proven in \cite{B-Smoothing}.
\end{rem}

\section*{Acknowledgments}
We are grateful to Anders Bj\"{o}rner, Elmar Vogt, Federico Ardila, Frank Lutz,  G\"{u}nter Ziegler, Tadeusz Januszkiewicz and Victor Chepoi, for useful suggestions. \\
Karim~Adiprasito acknowledges support by a Minerva fellowship of the Max Planck Society, an NSF Grant DMS 1128155, an ISF Grant 1050/16 and ERC StG 716424 - CASe.\\
Bruno~Benedetti acknowledges support by NSF Grants 1600741 and 1855165, the DFG Collaborative Research Center TRR109, and the Swedish Research Council VR 2011-980. 
\\Part of this work was supported by the National Science Foundation under Grant No. DMS-1440140 while the authors were in residence at the Mathematical Sciences Research Institute in Berkeley, California, during the Fall 2017 semester.

{\small
\def\cprime{$'$}
\providecommand{\bysame}{\leavevmode\hbox to3em{\hrulefill}\thinspace}
\providecommand{\MR}{\relax\ifhmode\unskip\space\fi MR }
\providecommand{\MRhref}[2]{%
  \href{http://www.ams.org/mathscinet-getitem?mr=#1}{#2}
}
\providecommand{\href}[2]{#2}

}

\end{document}